\documentclass[reqno,11pt]{amsart}
\usepackage{amssymb}
\usepackage[initials]{amsrefs}

\newtheorem{theorem}{Theorem}[section]
\newtheorem{lemma}[theorem]{Lemma}
\newtheorem{remark}[theorem]{Remark}
\newtheorem{proposition}[theorem]{Proposition}
\newtheorem{conjecture}[theorem]{Conjecture}

\DeclareMathOperator{\NBF}{\mathcal{NBF}}
\DeclareMathOperator{\NBP}{\mathcal{NBP}}
\DeclareMathOperator{\NBL}{\mathcal{NBL}}
\DeclareMathOperator{\card}{card}

\begin{document}

\title[Characterizing continuity for polynomials]{Bounded and unbounded polynomials and multilinear forms: Characterizing continuity}

\author[J. L. G\'{a}mez-Merino et al]{Jos\'{e} L. G\'{a}mez-Merino \and Gustavo A. Mu\~{n}oz-Fern\'{a}ndez \and Daniel Pellegrino \and Juan B.~Seoane-Sep\'{u}lveda}

\address{Departamento de An\'{a}lisis Matem\'{a}tico,\newline\indent Facultad de Ciencias Matem\'{a}ticas, \newline\indent Plaza de Ciencias 3, \newline\indent Universidad Complutense de Madrid,\newline\indent Madrid, 28040, Spain.}
\email{jlgamez@mat.ucm.es \newline \phantom{E-mail address: }gustavo\_fernandez@mat.ucm.es \newline \phantom{E-mail address: }jseoane@mat.ucm.es}

\address{Departamento de Matem\'{a}tica, \newline\indent Universidade Federal da Para\'{\i}ba, \newline\indent 58.051-900 - Jo\~{a}o Pessoa, Brazil.} \email{dmpellegrino@gmail.com}

\thanks{J.L. G\'{a}mez-Merino, G. A. Mu\~{n}oz-Fern\'{a}ndez and J. B. Seoane-Sep\'{u}lveda were supported by the Spanish Ministry
of Science and Innovation, grant MTM2009-07848.}

\begin{abstract}
In this paper we prove a characterization of continuity for polynomials on a normed space. Namely, we prove that a polynomial is continuous if and only if it maps compact sets into compact sets. We also provide a partial answer to the question as to whether a polynomial is continuous if and only if it transforms  connected sets into connected sets. These results motivate the natural question as to how many non-continuous polynomials there are on an infinite dimensional normed space. A problem on the \emph{lineability} of the sets of non-continuous polynomials and multilinear mappings on infinite dimensional normed spaces is answered.
\end{abstract}


\maketitle

\section{Introduction and notation}

It is well-known (see \cite[Theorem 2]{Velleman}) that a mapping
$f:{\mathbb R}\rightarrow {\mathbb R}$ is continuous if and only
if it satisfies the following two conditions:

\begin{enumerate}
\item $f$ maps compact sets into compact sets.\label{cond:1}
\item $f$ maps connected sets into connected sets.\label{cond:2}
\end{enumerate}

At the other end of the scale, it is possible to construct
$2^{\mathfrak c}$-dimensional spaces of everywhere discontinuous
functions in ${\mathbb R}^{\mathbb R}$ satisfying only
one of the above conditions (see \cite{GMS}). However, the same
situation does not hold for the case of polynomials on a normed
space. Actually, condition \eqref{cond:1} characterizes the continuity of a polynomial on a
normed space, which is proved in Section \ref{Sec:2}. As we will also see in Section 2, we study when condition \eqref{cond:2} above characterizes continuity for polynomials on normed spaces, problem which will be solved partly.

Finally, Section \ref{Sec:3} is devoted to the construction of
linear spaces of maximal dimension of non-bounded polynomials
between normed spaces.

For convenience we recall the basic definitions and standard results needed to discuss
polynomials on normed spaces. A map $P:E\rightarrow F$ is an
{\em $n$-homogeneous polynomial} if there is a
symmetric $n$-linear mapping $L:E^n \rightarrow F$ for which
$P(x)=L(x,\ldots,x)$ for all $x\in E$. In this case it is
convenient to write $P={\widehat L}$.

We let ${\mathcal P}_a(^{n}E;F)$, ${\mathcal L}_a(^{n}E;F)$ and
${\mathcal L}_a^{s}(^{n}E;F)$ denote respectively the linear spaces
of all $n$-homogeneous polynomials from $E$ into $F$, the
$n$-linear mappings from $E$ into $F$ and the
symmetric $n$-linear mappings from $E$ into $F$. More generally, a map
$P:E\rightarrow F$ is a {\em polynomial of degree at most $n$} if
    $$
    P=P_0 +P_1 +\cdots +P_n,
    $$
where $P_k \in {\mathcal
P}_a(^{k}E;F)$ $(1\leq k\leq n)$, and $P_0:E\rightarrow F$ is a
constant function. The polynomials of degree at most $n$ between the normed spaces $E$ and $F$ are denoted by ${\mathcal P}_{n,a}(E;F)$.

Polynomials on a finite dimensional normed space are always continuous; however, the same thing does not happen for infinite dimensional normed spaces. Boundedness is a characteristic property of continuous polynomials on a normed space. In particular, $P\in{\mathcal P}_{n,a}(E;F)$ is continuous if and only if $P$ is bounded on the unit ball of $E$ (denoted by ${\mathsf B}_E$). This is standard and particularly well-known for homogeneous polynomials (see for instance \cite[Proposition 1.11]{Dineen}). For the non-homogeneous case, a complexification procedure lets us focus our attention on polynomials defined on a complex normed space. Let $P$ be a polynomial of degree at most $n$ on the complex normed space $E$. We define the homogenization of $P$ by
    $$
    Q(x,\lambda)=\begin{cases}
    \lambda^n P\left(\frac{x}{\lambda}\right)&\text{if $\lambda\ne 0$,}\\
    0&\text{if $\lambda=0$,}
    \end{cases}
    $$
for every $(x,\lambda)\in E\oplus {\mathbb C}$. It is a simple exercise to prove that $Q$ is a homogeneous polynomial on $E\oplus {\mathbb C}$. Let $E\oplus_\infty{\mathbb C}$ stand for $E\oplus{\mathbb C}$ endowed with the norm $\|(x,\lambda)\|_\infty=\max\{\|x\|,|\lambda|\}$. Now if $P$ is bounded on ${\mathsf B}_E$, by the Maximum Modulus Principle
    \begin{align*}
    \sup\{\|Q(x,\lambda)\|:\|(x,\lambda)\|_\infty\leq 1\}&=
    \sup\left\{\left\|\lambda^n P\left(\frac{x}{\lambda}\right)\right\|:\|x\|\leq 1,\ |\lambda|\leq 1\right\}\\
    &=\sup\left\{\left\| P\left(\frac{x}{\lambda}\right)\right\|:\|x\|\leq 1,\ |\lambda|= 1\right\}\\
    &=\sup\left\{\left\| P(x)\right\|:\|x\|\leq 1\right\}.
    \end{align*}
Hence $Q$ is bounded on $E\oplus_\infty{\mathbb C}$, and therefore continuous. This implies that $P$ is also continuous since $P$ is a restriction of $Q$. Conversely, if $P$ is continuous, $Q$ is clearly continuous for all $(x,\lambda)\in E\otimes_\infty {\mathbb C}$ with $\lambda\ne 0$. Thus $Q$ is continuous in $E\otimes_\infty {\mathbb C}$ (see again \cite[Proposition 1.11]{Dineen}) and bounded in ${\mathsf B}_{E\otimes_\infty {\mathbb C}}$. Therefore $P$ must be bounded too in ${\mathsf B}_E$.

If $P:E\rightarrow F$ and $L:E^n \rightarrow F$ are, respectively, a continuous polynomial of degree at most $n$ and a continuous $n$-linear mapping we define
    \begin{align*}
    \|P\|&=\sup\{\|P(x)\|:\ \|x\|\leq 1\},\\
    \|L\|&=\sup  \{\|L(x_1,\ldots,x_n)\|:\ \|x_1\|\leq 1,\ldots , \|x_n\|\leq 1\}.
    \end{align*}

We let ${\mathcal P}(^{n}E;F)$, ${\mathcal P}_n(E;F)$, ${\mathcal L}(^{n}E;F)$ and
${\mathcal L}^{s}(^{n}E;F)$ denote, respectively, the normed spaces
of  the continuous $n$-homogeneous polynomials from $E$ into $F$, the continuous polynomials of degree at most $n$ from $E$ into $F$, the
continuous $n$-linear mappings from $E$ into $F$, and the
continuous symmetric $n$-linear mappings from $E$ into $F$.

In general the results on the continuity of scalar-valued polynomials and multilinear forms can be easily extended to vector-valued polynomials and multilinear mappings.
If ${\mathbb K}$ is the real or complex field we use the notations
${\mathcal P}(^{n}E)$, ${\mathcal P}_n(E)$, ${\mathcal L}(^{n}E)$ and ${\mathcal
L}^{s}(^{n}E)$ in place of ${\mathcal P}(^{n}E;{\mathbb K})$, ${\mathcal P}_n(E;{\mathbb K})$,
${\mathcal L}(^{n}E;{\mathbb K})$, and ${\mathcal
L}^{s}(^{n}E;{\mathbb K})$ respectively.

\section{A characterization of continuity for polynomials}\label{Sec:2}

In this section we will consider both conditions \eqref{cond:1} and \eqref{cond:2} given in the
Introduction, in the frame of polynomials on normed spaces. Let us begin with proving that, actually,
condition \eqref{cond:1} characterizes the continuity of polynomials on any normed space.

\begin{theorem}\label{char1}
If $E$ is a normed space and $P$ is a polynomial  on $E$
then $P$ is continuous if and only it transforms compact sets into
compact sets.
\end{theorem}

\begin{proof}
All continuous functions between topological spaces map compact sets into compact sets, so we just need to
prove that if $P$ maps compact sets into compact sets, then $P$ is
continuous. Actually, we only need to show that all polynomials mapping compact sets in compact sets are continuous
at $0$. If we prove that and $x_0\in E$ is arbitrary, then the
polynomial defined by $Q(x)=P(x+x_0)$ for all $x\in E$ also maps
compact sets into compact sets. Being $Q$ continuous at $0$, we
would also have that $P$ is continuous at $x_0$. Actually, a more general statement can be proved: a polynomial is continuous if and only if it is continuous at $0$.

Let us prove then that $P$ is continuous at $0$. Let $(x_k)$ be a
convergent sequence in $E\setminus\{0\}$ to $0$ such that $\lim_{k\rightarrow
\infty}P(x_k)$ does not exist or it is not equal to $P(0)$. Since
the set $C=\{x_k:k\in{\mathbb N}\}\cup\{0\}$ is compact and $P(C)$ is compact too by hypothesis, we can assume without loss of generality that $(P(x_{k}))$ converges to $a\ne P(0)$ and that $P(x_k)\ne P(0)$ for all $k\in{\mathbb N}$.

Observe that only one of the following statements can hold:

\begin{enumerate}
\item  There exists a subsequence $(y_k)$ of $(x_n)$ such that $P$ is injective on $\{y_k:k\in{\mathbb N}\}$.

\item $P(x_k)=a$ for all but a finite number of $k$'s.
\end{enumerate}

For the first case consider the set $C^*=C\setminus P^{-1}(a)$,
which is compact. However $P(C^*)$ is not even closed since it does not contain
its limit point $a$.

For the second case we may assume that $P(x_k)=a$ for
all $k\in{\mathbb N}$. Now suppose $P=P_n+P_{n-1}+\cdots+P_1+P_0$, where
$P_j\in{\mathcal P}_{a}(^jE)$ and $P_0$ is a constant function taking the value $P(0)$.  Then for each $k\in{\mathbb N}$,
$P_j(x_k)$ cannot vanish for every $j=1,\ldots,n$ (otherwise $P(x_k)=P(0)$). Therefore
the one variable polynomial defined by $p_k(\lambda):=P(\lambda x_k)$, for all
$\lambda\in {\mathbb R}$, is not constant, and hence it takes
infinitely many values on every interval. Using the continuity of the
polynomial $p_k$ one can construct a sequence $(\lambda_k)\subset(0,1]$
such that for each $k\in{\mathbb N}$ we have
    $$
    |P(\lambda_kx_k)-P(x_k)|=|p_k(\lambda_k)-p_k(1)|<\frac{1}{k}
    $$
and
    $$
    P(\lambda_kx_k)\notin\{P(0),P(\lambda_1x_1),\ldots,P(\lambda_{k-1}x_{k-1})\}.
    $$
Notice that
    $$
    |P(\lambda_kx_k)-a|\leq |P(\lambda_kx_k)-P(x_k)|+|P(x_k)-a|\longrightarrow 0\quad\text{as $k\rightarrow\infty$.}
    $$
Finally, by letting $y_k=\lambda_k x_k$, we have that $(y_k)\subset E\setminus\{0\}$, $P(y_k)\ne P(0)$, $\lim_{k\rightarrow\infty}y_k=0$, $\lim_{k\rightarrow\infty}P(y_k)=a$, and $P$ is injective over $\{y_k:k\in{\mathbb N}\}$. This leads us to a contradiction as in the first case.
\end{proof}

After checking that condition \eqref{cond:1} from the Introduction characterizes continuity, a natural question arises now:
\begin{center}
\begin{quote}
{\em Is $P\in{\mathcal P}_{n,a}(E)$ continuous if and only if for every connected set $C\in E$, $P(C)$ is also connected for every infinite dimensional normed space $E$?}
\end{quote}
\end{center}

Unfortunately, this general question seems much deeper than it looks at first sight, although we can prove it for the particular case of homogeneous polynomials of degree 1 and 2, as we see next:
\begin{proposition}\label{rem}
Let $P\in{\mathcal P}(^nE)$ with $n=1,2$. Then $P$ is continuous if and only if it transforms connected sets into connected sets.
\end{proposition}

\begin{proof}
If $P$ is continuous, it obviously transforms connected sets into
connected sets. Now suppose $P$ is not continuous. Then there
exists a sequence of non null vectors $\{x_n\}$ such that
$\lim_{k}x_k=0$ but $\lim_k P(x_k)=\infty$. We can also choose the
$x_k$'s so that $\{P(x_k)\}$ is an increasing sequence and
$P(x_1)>0$.

Now consider the connected set
$C=\left(\bigcup_{k=1}^\infty [x_k,x_{k+1}]\right)\cup \{0\}$,
where $[x_k,x_{k+1}]$ is the segment with endpoints $x_k$ and
$x_{k+1}$ for every $k\in{\mathbb N}$. If $n=1$, by linearity $P([x_k,x_{k+1}])=[P(x_k),P(x_{k+1})]$ for all $k\in{\mathbb N}$. Hence $P(C)=[P(x_1),\infty)\cup\{0\}$ and since $P(x_1)>0$, $P(C)$ is not connected. Furthermore, if $n=2$ and $L\in{\mathcal L}^s(^2E)$ is the polar of $P$, we can assume that $L(x_k,x_{k+1})\geq 0$. Indeed, we
just need to replace $x_k$ by $-x_k$ if necessary. It is important
to notice that $P(x_k)=P(-x_k)$.  Since
    \begin{align*}
    P(\lambda x_n+(1-\lambda)
    x_{k+1})&=\lambda^2P(x_k)+2\lambda(1-\lambda)L(x_k,x_{k+1})+(1-\lambda)^2P(x_{k+1})\\
    &\geq \lambda^2P(x_k)+(1-\lambda)^2P(x_{k+1})\\
    &\geq \left[\lambda^2+(1-\lambda)^2\right]P(x_k)\\
    &\geq P(x_k),
    \end{align*}
for every $\lambda \in[0,1]$, we have that
$P([x_k,x_{k+1}])\subset [P(x_k),\infty)$. This, together with the
fact that $\lim_kP(x_k)=\infty$ imply that
$P(C)=[P(x_1),\infty)\cup\{0\}$. Finally, since $P(x_1)>0$, $P(C)$ is not connected.
\end{proof}

\begin{conjecture}
It is our belief that condition \eqref{cond:2} also characterizes continuity for arbitrary polynomials on any infinite dimensional normed space.
\end{conjecture}

\begin{remark}
Although we do not know the answer to the previous conjecture, we do know that if $L\in{\mathcal L}_a(^nE)$ transforms connected sets in $E^n$ into connected sets, then it is continuous. Indeed, using Proposition \ref{rem} with $n=1$, it is easy to see that $L$ is separately continuous, and hence continuous. 
\end{remark}

\section{Non-bounded multilinear mappings and polynomials}\label{Sec:3}

After learning the characterizations obtained in the previous section (Theorems \ref{char1} and Proposition \ref{rem}), this section is devoted to the relatively new notion of {\em lineability}, which will tie the paper together. This notion of lineability has the following motivation: Take a function with some special or pathological property. Coming up with a concrete example of such a function can be a difficult task. Actually, it may seem that if one succeeds in finding one example of such a function, one might think that there cannot be too many functions of that kind. Probably one cannot even find infinite dimensional vector spaces of such functions. This is, however, exactly what has happened. The search for large algebraic structures of functions with pathological properties has lately become somewhat of a new trend in mathematics. Let us recall that a set $M$ of functions satisfying some pathological property is said to be \emph{lineable} if $M\cup\{0\}$ contains an infinite dimensional vector space.  More specifically, we will say that $M$ is \emph{$\mu$-lineable} if $M\cup\{0\}$ contains a vector space of dimension $\mu$, where $\mu$ is a cardinal number. We refer to the interested reader to \cite{GMS,sierpinski,ACPS,AGPS,AronGurariySeoane,APS,AS,Bernal,BDFP,BDP,BMP,GGMS,MPPS,PS} for recent advances in this theory.

If $E$ is a normed space, in this section $\NBL(^nE)$,  $\NBL^s(^nE)$,
$\NBP(^nE)$ and $\NBP_n(E)$ represent, respectively, the set of non-bounded linear forms on $E$, the set of non-bounded symmetric $n$-linear forms on $E$, the set of non-bounded scalar-valued $n$-homogeneous polynomials on $E$ and the set of non-bounded scalar-valued polynomials on $E$ of degree at most $n$. Our results on the lineability of $\NBL(^nE)$,  $\NBL^s(^nE)$,
$\NBP(^nE)$ and $\NBP_n(E)$ rely on the lineability of the set of non-bounded scalar-valued functions defined on an infinite set $I$, denoted by $\NBF(I)$. The following set-theoretical lemma (see \cite[Lemma 4.1]{AronGurariySeoane}) will be needed for our main result in this section.

\begin{lemma}\label{sets}
If $C_1,\ldots,C_m$ are $m$ arbitrary, different, non-empty sets,
then there exists $k\in\{1,\ldots,m\}$ such that for every $1\leq
j\leq m$ with $j\ne k$, we have that $C_k\backslash C_j\ne
\varnothing$.
\end{lemma}

Also, the next lemma (although of independent interest in itself) will be necessary.

\begin{lemma}\label{nbfunctions}
If $I\subset {\mathbb R} $ is uncountable, then the set $\NBF(I)$ is $2^{\card(I)}$-lineable.
\end{lemma}

\begin{proof}
For each non-void $C\subset I$ let $H_C\colon{\mathbb R}\times I^{\mathbb
N}\to {\mathbb R}$ be defined by
    $$
    H_C(x,x_1,\ldots,x_j,\ldots)=x\cdot \prod_{j=1}^\infty
    \chi_C(x_j).
    $$
If we fix a sequence $(x_n)\subset C$ then
$H_C(x,x_1,\ldots,x_n,\ldots)=x$ for all $x\in {\mathbb R}$, and
hence the $H_C$'s are not bounded. Moreover, if $C_1,\ldots,C_m$ are $m$
different subsets of $I$ and $\sum_{k=1}^m\lambda_k
H_{C_k}$ is a linear combination of the $H_{C_k}$'s ($1\leq k\leq m$) with $\lambda_k\ne 0$ for
all $k=1,\ldots,m$ then, renaming the sets if necessary, from Lemma \ref{sets} it follows that for each $1\leq j< m$ there exists
$x_j\in C_m\backslash C_j$. Now let
$v=(x,x_1,x_2,\ldots,x_{m-1},x_{m-1},\ldots)\in {\mathbb
R}\times I^{\mathbb N}$ with $x\in{\mathbb R}$ arbitrary. Then
    $$
    \sum_{k=1}^m\lambda_k
    H_{C_k}(v)=\sum_{k=1}^m\lambda_k\left[x\prod_{j=1}^{m-1}\chi_{C_k}(x_j)\right]=\lambda_mx,
    $$
for all $x\in{\mathbb R}$, which shows that $\sum_{k=1}^m\lambda_k H_{C_k}$ is not bounded.

Now if $\sum_{k=1}^m\lambda_k H_{C_k}\equiv 0$ and we set $v=(1,x_1,x_2,\ldots,x_{m-1},x_{m-1},\ldots)\in {\mathbb
R}\times I^{\mathbb N}$, then
        $$
    0=\sum_{k=1}^m\lambda_k
    H_{C_k}(v)=\sum_{k=1}^m\lambda_k\left[\prod_{j=1}^{m-1}\chi_{C_k}(x_j)\right]=\lambda_m,
    $$
which contradicts the fact that $\lambda_k\ne 0$ for all
$k=1,\ldots,m$. Finally, if $\Phi:I\leftrightarrow {\mathbb R}\times
I^{\mathbb N}$ is a bijection, then the set $\{H_C\circ
\Phi:C\subset I\}$ has unbounded non trivial linear combinations and it is linearly independent with cardinality
$2^{\card(I)}$, which concludes the proof.
\end{proof}

We are now ready to state and prove the main (and general) lineability result in this section:

\begin{theorem}
If $n\in{\mathbb N}$ and $E$ is a normed space of infinite dimension $\lambda$ then the sets $\NBL(^nE)$,  $\NBL^s(^nE)$,
$\NBP(^nE)$ and $\NBP_n(E)$ are $2^\lambda$-lineable.
\end{theorem}

\begin{proof}
Let $\{e_i:i\in I\}$ be a basis for $E$ with $\card(I)=\lambda$ of norm $1$ vectors.
By Lemma \ref{nbfunctions} there exists $2^\lambda$ linearly independent mappings $\{f_j:j\in J\}$ (with $\card(J)=2^\lambda$) generating a  linear space of unbounded real valued functions on $I$. If
for each $j\in J$ we define a multilinear mapping
$L_j:E\rightarrow {\mathbb R}$ by
    \begin{equation}\label{Ljs}
    L_j(e_{i_1},\ldots,e_{i_n})=f_j(i_1)+\dotsb+f_j(i_n),
    \end{equation}
for all choices of $(i_1,\ldots,i_n)\in I^n$, then
$\{L_j:j\in J\}$ is a linearly independent set in $\NBL(^nE)$. Indeed if
$\sum_{k=1}^m\lambda_kL_{j_k}\equiv 0$ with
$\lambda_1,\ldots,\lambda_m\in{\mathbb R}$, then
for every $i\in I$ we have
    $$
    n\sum_{k=1}^m\lambda_kf_{j_k}(i)=\sum_{k=1}^m\lambda_kL_{j_k}(e_i,\stackrel{(n)}{\ldots},e_i)=0,
    $$
from which $\sum_{k=1}^m\lambda_kf_{j_k}(i)=0$ for every $i\in I$.
In other words $\sum_{k=1}^m\lambda_kf_{j_k}\equiv 0$ and
therefore $\lambda_k=0$ for all $1\leq k\leq m$ since the
$f_{j_k}$'s are linearly independent.

On the other hand, if $\lambda_k\ne 0$ for $k=1,\ldots,m$, then
    $$
    \left\|\sum_{k=1}^m\lambda_kL_{j_k}(e_i,\stackrel{(n)}{\ldots},e_i)\right\|=n\left|\sum_{k=1}^m\lambda_kf_{j_k}(i)\right|.
    $$
Hence $\sum_{k=1}^m\lambda_kL_{j_k}$ is not bounded since $\sum_{k=1}^m\lambda_kf_{j_k}$ is not bounded either.
This shows that $\NBL(^mE)$ is $2^\lambda$-lineable.

In order to prove that $\NBL^s(^nE)$ is $2^\lambda$-lineable, consider the set $\{\overline{L}_j:j\in J\}$, where the $L_j$'s are as in \eqref{Ljs} and $\overline{L}_j$ is the symmetrization of $L_j$ for all $j\in J$. If
$\sum_{k=1}^m\lambda_k\overline{L}_{j_k}\equiv 0$ with
$\lambda_1,\ldots,\lambda_m\in{\mathbb R}$, then
for every $i\in I$ we have
    $$
    n\left(\sum_{k=1}^m\lambda_kf_{j_k}(i)\right)=\sum_{k=1}^m\lambda_kL_{j_k}(e_i,\stackrel{(n)}{\ldots},e_i)=
    \sum_{k=1}^m\lambda_k\overline{L}_{j_k}(e_i,\stackrel{(n)}{\ldots},e_i)=0,
    $$
from which $\sum_{k=1}^m\lambda_kf_{j_k}(i)=0$ for every $i\in I$.
In other words $\sum_{k=1}^m\lambda_kf_{j_k}\equiv 0$ and
therefore $\lambda_k=0$ for all $1\leq k\leq m$ since the
$f_{j_k}$'s are linearly independent.

If now $\lambda_k\ne 0$ for all $k=1,\ldots m$, then
    $$
    \left\|\sum_{k=1}^m\lambda_k\overline{L}_{j_k}(e_i,\stackrel{(n)}{\ldots},e_i)\right\|=
    \left\|\sum_{k=1}^m\lambda_kL_{j_k}(e_i,\stackrel{(n)}{\ldots},e_i)\right\|=n\left|\sum_{k=1}^m\lambda_kf_{j_k}(i)\right|,
    $$
and since $\sum_{k=1}^m\lambda_kf_{j_k}$ is not bounded, then $\sum_{k=1}^m\lambda_k\overline{L}_{j_k}$ is not bounded either. Therefore $\NBL^s(^nE)$ is $2^\lambda$-lineable.

As a corollary to the fact that $\NBL^s(^nE)$ is $2^\lambda$-lineable, we deduce that $\NBP(^nE)$ is also $2^\lambda$-lineable since the algebraic spaces ${\mathcal L}^s_a(^nE)$ and  ${\mathcal P}_a(^nE)$ are isomorphic.

Finally, $\NBP_n(E)$ is $2^\lambda$-lineable since $\NBP(^nE)\subset \NBP_n(E)$ and $\NBP(^nE)$ is $2^\lambda$-lineable
\end{proof}

\noindent\textbf{Acknowledgements.} The authors express their gratitude to Prof. D. Azagra for fruitful conversations and for pointing out Proposition \ref{rem}.

\end{document}